\documentclass[12pt]{amsart}
\usepackage{amsmath, amsthm, amssymb}
\usepackage{amssymb}
\usepackage{amsfonts}
\usepackage{amscd}
\usepackage{mathrsfs}
\usepackage{latexsym}
\usepackage{amsmath}
\usepackage{txfonts}
\usepackage{tikz}
\usepackage{graphicx}
\usepackage{subfig}

\usepackage[all, cmtip]{xy}

\newtheorem{theorem}{Theorem}[section]
\newtheorem{lemma}[theorem]{Lemma}
\newtheorem{corollary}[theorem]{Corollary}
\newtheorem{proposition}[theorem]{Proposition}

\theoremstyle{definition}

\newtheorem{example}[theorem]{Example}

\theoremstyle{remark}
\newtheorem{claim}[theorem]{Claim}

\newtheorem{remark}[theorem]{Remark}
\newtheorem{question}[theorem]{Question}

\newtheorem{assumption}[theorem]{Assumption}
\newcommand{\sshf}[1]{\mathcal{O}_{#1}}
\newcommand{\shf}[1]{\mathcal{#1}}

\newcommand{\prj}[1]{\mathbb{P}^{#1}}

\newcommand{\iso}{\simeq}
\newcommand{\ses}[3]{0\rightarrow#1\rightarrow#2\rightarrow#3\rightarrow{0}}
\newcommand{\is}[1]{\mathscr{I}_{#1}}

\newcommand{\paren}[1]{\left(#1\right)}

\numberwithin{equation}{section}

\begin{document}
\allowdisplaybreaks
\title{Singularities of secant varieties}

%    Information for first author
\author{Chih-Chi Chou}
\author{Lei Song}

%    Address of record for the research reported here
\address{Department of Mathematics, University of Washington, Seattle, WA 98195}
%    Current address
%\curraddr{Department of Mathematics and Statistics,
%University of Illinois at Chicago, Chicago, IL 60607}
\email{cchou9@math.washington.edu}

\address{Department of Mathematics, University of Kansas, Lawrence, KS 66045}
\email{lsong@ku.edu}
%    \thanks will become a 1st page footnote.
%\thanks{The first author was supported in part by NSF Grant \#000000.}

%\date{Feb 19th, 2015}

\dedicatory{}


\begin{abstract}
We study the singularities of the secant variety $\Sigma(X,L)$ associated to
a smooth variety $X$ embedded by a sufficiently positive adjoint line bundle $L$.
We show that $\Sigma(X,L)$ is always Du Bois singular. Examples of secant varieties with worse singularities when $L$ has weak positivity are provided. We also give a necessary and sufficient condition for $\Sigma(X, L)$ to have rational singularities.
\end{abstract}

\maketitle

\section{Introduction}

Given a smooth projective variety $X$ of dimension $n$ over an algebraically closed field of characteristic zero, and a very ample line bundle $L$ that embeds $X$ into a projective space $\prj{N}$. The (first) secant variety $\Sigma(X,L)$ is defined as the Zariski closure of the union of 2-secant lines to $X$ in $\mathbb{P}^{N}$.
Recently, Ullery \cite{Ullery14} gave a sufficient condition on $L$ for the normality of the secant variety $\Sigma(X, L)$, completing the results of Vermeire.
She showed that, among other things, when $X$ is a curve, $\Sigma(X,L)$ is normal if $\deg{L}\ge 2g+3$; when $n\ge 2$, $\Sigma(X,L)$ is normal
if $L=\omega_X\otimes A^{2(n+1)}\otimes B$ , where $A,B$
are very ample and nef line bundles, respectively.

Inspired by the paper \cite{Ullery14}, we study the singularities of $\Sigma(X, L)$ from the cohomological point of view.
To state results in a uniform way,
throughout the paper we make the following assumption on $L$ unless otherwise stated (In Section 5, the situation when $L$ has weaker positivity is discussed).
\begin{assumption}\label{assumption}
For $n\ge 2$, we assume that $L=\omega_X\otimes A^{2(n+1)}\otimes B$, where $A$ is very ample and $B$ is a nef line bundles respectively. For $X$ being a curve of genus $g$, we assume that $\deg{L}\ge 2g+3$.
\end{assumption}
According to a result of Ein and Lazarsfeld \cite{EL93} (in case $n\ge 2$), such $L$ satisfies Property $N_{n+1}$, i.e.~$X$ embeds in $\prj{N}$ under the complete linear system $|L|$ as a projectively normal variety, the homogeneous ideal of $X$ is generated by quadrics, and the minimal graded free resolution of $\sshf{X}$ is linear up to $(n+1)$-th step. One may expect that the singularities of $\Sigma(X, L)$ will be somewhat well behaved if $L$ satisfies the assumption.
Our first result confirms this expectation in the sense that
\begin{theorem}(cf.~Theorem \ref{Du Bois})\label{main theorem}
Under Assumption \ref{assumption}, $\Sigma(X, L)$ has Du Bois singularities.
\end{theorem}
The notion of Du Bois singularities is originated from complex geometry and plays an important role in classification of algebraic varieties as shown in \cite{KK}.
Roughly speaking, a variety has Du Bois singularities if its cohomological behavior is the same as a simple normal crossing variety.
From this point of view, the notion of Du Bois singularities is a generalization
of rational singularities.

As a result, it is natural to ask whether a secant variety has rational singularities in general. Before answering this question, we recall the Kempf's criterion for rational singularities: given a normal variety $Z$ and a resolution of singularities $f: Y\rightarrow Z$,
then $Z$ has rational singularities if and only if $Z$ is Cohen-Macaulay and
$f_*\omega_Y=\omega_Z$.
So we prove the following two theorems.
\begin{theorem}(= Theorem \ref{CM})\label{CM1}
Under Assumption \ref{assumption}, for any closed point $x\in X \subset \Sigma(X,L)$ we have
\begin{equation*}
    \text{depth}_x\Sigma(X, L)=n+2+\max\{i \:|\: i\le n-1, \text{and}\; H^j(X, \sshf{X})=0,  1\le j\le i\}
\end{equation*}
Here we adopt the convention that if the set $\{i \:|\: i\le n-1, \text{and}\; H^j(X, \sshf{X})=0,  1\le j\le i\}$ is empty, then the max is 0.
\end{theorem}
Note that with the assumption on $L$, $\Sigma(X,L)\backslash X$ is smooth
and dim$\Sigma(X,L)=2n+1$, so it follows from Theorem \ref{CM1} that $\Sigma(X,L)$ is Cohen-Macaulay if and only if $H^i(X, \mathcal{O}_X)=0$, for all $1\le i \le n-1$.
On the other hand, the second condition of rational singularities
is controlled by the top cohomology group as the following shows.
\begin{theorem}(= Theorem \ref{pushforward of dualizing sheaf})\label{push dualizing sheaf}
Under Assumption \ref{assumption}.
Let $t:\mathbb{P}(E_L)\rightarrow \Sigma(X,L)$ be the natural resolution of singularities (see \S 2). Then $t_*\omega _{\mathbb{P}(E_L)}\iso \omega_{\Sigma(X, L)}$ if and only if $H^n(X, \mathcal{O}_X)=0$.
\end{theorem}
In Theorem \ref{push dualizing sheaf}, we do not need
$\Sigma(X,L)$ to be Cohen-Macaulay.
Here $\omega_{\Sigma(X,L)}$ is the $-(2n+1)$-th cohomology of the dualizing complex $\omega ^{\bullet}_{\Sigma(X,L)}$.
Combing the two theorems above, we conclude with a corollary, which in case $n=1$ was observed by Vermeire \cite {Vermeire08}.
\begin{corollary}
Under Assumption \ref{assumption}. $\Sigma(X,L)$ has rational singularities if and only if
$H^i(X, \mathcal{O}_X)=0$, for all $1\le i \le n$.\qed
\end{corollary}

Another natural question to ask is that if the embedding line bundle $L$ does not satisfy Assumption \ref{assumption}, is the secant variety $\Sigma(X,L)$
still Du Bois?
This is not the case. In Section 4, we give examples of the secant varieties are normal but not Du Bois.

In the minimal model program, typical types of singularities considered  are Kawamata log terminal (klt) and log canonical (lc). And it is well known that klt and lc implies rational and Du Bois singularities, respectively. (cf.~\cite{kollarMori08} and \cite{KK}).
So we ask the following question:

\begin{question}\label{Question}
Can one find some boundary divisor $\Delta$ on $\Sigma(X,L)$ such that $\paren{\Sigma(X,L), \Delta}$ is a lc pair, or even more, klt in the case that $\Sigma(X, L)$ has rational singularities?
\end{question}

As an example (cf.~Theorem \ref{log canonical}) we show that when $X\iso \mathbb{P}^2$, $\Sigma(X,L)$ is log canonical with some boundary $\Delta$.

Singularities of classical varieties are of great interest to
algebraic geometers.
For example, it has been shown that generic determinantal varieties, Schubert varieties and Richardson varieties are klt (\cite{IV}, \cite{JS}, \cite{AS} and \cite{Kumar}). In particular,
they all have rational singularities.
The result of this paper provides examples of classical varieties having
Du Bois but not rational singularities.

Besides the singularities mentioned above, people are also interested in
the notions of singularities from the view point of arc spaces. Recently Mather-Jacobian (MJ) discrepancy and related singularities are introduced and have been explored by several authors. W. Niu pointed out to us that the answer to Question \ref{Question} is negative in the context of MJ discrepancy: if $L$ is sufficiently positive, then for any point $x\in X$, $\dim T_x\Sigma(X, L)=\dim T_x\mathbb{P}^N\gg 2n$, so by \cite[Proposition 3.3]{EinIshii13}, $\Sigma(X, L)$ cannot be MJ-log canonical.

{\it Acknowledgments}:
The authors would like to thank Lawrence Ein, S\'{a}ndor Kov\'{a}cs, Wenbo Niu, Lorenzo Prelli for useful discussions, thank Brooke Ullery for her interest in this work. We would also like to thank the referees' corrections and constructive comments.

\section{Notation and preliminaries}

\subsection{Geometry of secant varieties}
We will use the constructions and notations in \cite{Ullery14} and recall them here
for the reader's convenience. We refer the interested reader to \cite{Bertram} and \cite{Vermeire} for
constructions in curves and higher dimension varieties respectively.

Let $X$ be a smooth variety of dimension $n$ embedded in $\mathbb{P}^{N}$ by a very ample line bundle $L$. From now on, we will simply write $\Sigma(X)$ or $\Sigma$ for $\Sigma(X,L)$ when the context is clear. $L$ is said to be \textit{$k$-very ample} if the natural map $H^0(L)\rightarrow H^0(L\otimes\sshf{\xi})$ is surjective for every length $k+1$, zero dimensional closed subscheme $\xi$ of $X$.
Consider a line bundle of the form $L=\omega_X\otimes A^{\otimes k}\otimes B$, where $A$ is very ample and $B$ is a nef line bundles respectively, and $k\ge 1$. As shown in \cite{Ullery14}, $L$ is 3-very ample for $k\ge n+4$, and $\Sigma$ is normal for $k\ge 2(n+1)$ in case $n\ge 2$. We remark that $L$ can fail to be 3-very ample if $k\le n+3$. An example is the Veronese embedding $j: \prj{2}\hookrightarrow\prj{5}$ under $|\sshf{\prj{2}}(2)|$. In this case $\dim\Sigma(\prj{2}, \sshf{\prj{2}}(2))< 5$ (cf.~\cite[p. 43]{Beauville96}), and hence $\sshf{\prj{2}}(2)$ cannot be 3-very ample.

Denote by $X^{[2]}$ the Hilbert scheme of two points on $X$. The universal family $\Phi\subset X^{[2]}\times X$ is isomorphic to the blowing up $\pi: Bl_{\Delta}{X^2}\rightarrow X^2$ over $X^2$, where $\Delta$ is the diagonal on $X^2$. Denote the two projections from $\Phi$ by $\sigma$ and $q$, as shown below

$$\xymatrix{
 &\Phi \ar[dl]_{\sigma}\ar[dr]^{q} &\\
X^{[2]} & & X.}$$
Here $q$ is the composition of $\pi$ and the first projection $p_1: X^2\rightarrow X$.

One can define the tautological vector bundle $E_L$ on $X^{[2]}$ by $E_L=\sigma_*(q^*L)$. The rank two bundle $E_L$ is tautological in the sense that for any length two closed subscheme $\xi$ of $X$,
$E_L\otimes k([\xi])=H^0(X, L\otimes\sshf{\xi})$, and $H^0(X, L)\iso H^0(X^{[2]}, E_L)$. When $L$ is very ample, $E_L$ is globally generated.

Now assume $L$ is 3-very ample. Then $\Sigma(X,L)$ is singular only along $X$. In this case, due to Bertram \cite{Bertram} in curve case and Vermeire \cite{Vermeire} in general, the projective bundle $\mathbb{P}(E_L)\subset X^{[2]}\times\mathbb{P}(H^0(X, L))$ together with second projection provides a natural resolution of singularities $t: \mathbb{P}(E_L) \rightarrow \Sigma(X,L)$.
The exceptional divisor of $t$ is isomorphic to $\Phi$.
In particular, given a closed point $x\in X$ the fiber $F_x$ of $t$ is isomorphic to $Bl_x{X}$, the blowing-up of $X$ at $x$.  As a result, one has the Cartesian diagram

\begin{equation}\label{Geometry of secant varieties}
    \xymatrix{
F_x  \ar[d]^{b_x} \ar@{^{(}->}[r] &\Phi \ar[d]^q \ar@{^{(}->}[r] &\mathbb{P}(E_L)\ar[d]^t\\
\{x\} \ar@{^{(}->}[r] &X \ar@{^{(}->}[r] &\Sigma(X, L).}
\end{equation}
The next easy lemmas will be used several times in the sequel.
\begin{lemma}\label{flat}
The map $q:\Phi \rightarrow X$ is smooth.
\end{lemma}
\begin{proof}
Since both $X$ and $\Phi$ are smooth and every fiber of $q$ has dimension $n$, $q$ is flat \cite[III, Ex.10.9]{Hartshorne77}.
For any closed point $y\in\Phi$, let $x=q(y)$. Then $\dim_{\kappa(y)}(\Omega_{\Phi/X}\otimes\kappa(y))=\dim_{\kappa(y)}(\Omega_{F_x/{\kappa(x)}}\otimes \kappa(y))=n$. Thus $\Omega_{\Phi/X}$ is locally free of rank $n$. So $q$ is smooth of relative dimension $n$.
\end{proof}

\begin{lemma}\label{higher direct images of structure sheaf}
For all $i\ge 0$,
\begin{equation*}
    R^iq_*\sshf{\Phi}\iso H^i(X, \sshf{X})\otimes\sshf{X}.
\end{equation*}

\end{lemma}
\begin{proof}Since $q=p_1\circ\pi$ and $R^j\pi_*\sshf{\Phi}=0$ for all $j>0$, one has $R^iq_*\sshf{\Phi}\iso {R^ip_1}_*\sshf{X^2}\iso H^i(X, \sshf{X})\otimes\sshf{X}$.
\end{proof}

\subsection{Du Bois singularities}

Recall that given a smooth complex variety $X$, the de Rham complex is  a resolution of constant sheaf $\mathbb{C}_X$,
\begin{equation*}
\mathcal{O}_X \rightarrow \Omega_X^1\rightarrow \Omega_X^2 \rightarrow \cdots\rightarrow\Omega^n_X.
\end{equation*}
For a singular variety, the Deligne-Du Bois complex $\underline{\Omega}^{\bullet}_X$ is a generalization of the de Rham complex.
The thorough description and properties of $\underline{\Omega}^{\bullet}_X$ are quite involved, we refer the interested reader to \cite[chapter 6]{kollar:book}.
There is a natural map
\begin{equation*}
\mathcal{O}_X\longrightarrow \underline{\Omega}^{0}_X:=Gr^0_{\text{filt}}\underline{\Omega}^{\bullet}_X,
\end{equation*}
and we say that $X$ has Du Bois singularities (DB for short) if the natural map is a quasi-isomorphism.
For example, a simple normal crossing reduced scheme is DB.

One can also define Du Bois singularities of a pair $(X,  Y)$  by considering the Deligne-Du Bois complex $\underline{\Omega}^{\bullet}_{X,Y}$.
Analogously, there is a natural map
\begin{equation*}
\is{Y/X} \longrightarrow \underline{\Omega}^{0}_{X,Y}:=Gr^0_{\text{filt}}\underline{\Omega}^{\bullet}_{X,Y},
\end{equation*}
where $\is{Y/X}$ is the ideal sheaf of $Y$ in $X$.
We say $(X,Y)$ is a Du Bois pair (DB pair for short) if the natural map is a quasi-isomorphism.
Note that being a DB pair does not imply either $X$ or $Y$ is DB, the notion is more about the ``relation" between $X$ and $Y$.
However if both $X$ and $Y$ are DB, then $(X,Y)$ is a DB pair.

\subsection{Local duality}

We recall local duality, which will be our main
tool for studying the dualizing sheaf of secant variety. Let $(R,p)$ be a local ring.
An injective hull $I$ of the residue field $k=R/p$ is an injective $R$-module
such that
for any non-zero submodule $N\subset I$ we have $N \cap k\neq 0$.
Or equivalently, the injective hull $I$ is the minimal injective module containing
$k$.

\begin{theorem}(cf.~\cite[$\S$V. Theorem 6.2]{Har66})\label{local duality}
Let $(R,p)$ be a local ring and $\mathcal{F}^{\bullet}\in D^{+}_{coh}(R)$.
Then
\begin{equation*}
\mathcal{R}\Gamma_p(\mathcal{F}^{\bullet})\rightarrow \mathcal{R}Hom
(\mathcal{R}Hom(\mathcal{F}^{\bullet}, \omega _R ^{\bullet}), I)
\end{equation*}
is an isomorphism.
\end{theorem}

In particular, taking $i$-th cohomology on both sides,
we have
\begin{equation}\label{equ}
\mbox{H}^i_p(\mathcal{F}^{\bullet})\iso Hom (\mbox{H}^{-i}(\mathcal{R}Hom(\mathcal{F}^{\bullet}, \omega _R ^{\bullet})),I)
\iso Hom (\mathcal{E}xt^{-i}(\mathcal{F}^{\bullet}, \omega _R ^{\bullet})),I).
\end{equation}

The $-i$ comes from switching the cohomology functor $\mbox{H}^i(\cdot)$ and $Hom(\:, I)$.
In this paper we will only use the case when $\mathcal{F}^{\bullet}$ is a module,
as in \cite[$\S$V. Corollary 6.3]{Har66}.
%Moreover, if we apply $Hom(\:, I)$ to Equation (\ref{equ}) we get
%\begin{equation}\label{equ2}
%Hom(\mbox{H}^i_p(\mathcal{F}^{\bullet}), I) \iso \mathcal{E}xt^{-i}(\mathcal{F}%^{\bullet}, \omega _R ^{\bullet}).
%\end{equation}
%(See \S V. Proposition 5.1 in \cite{Har66} )

%%%We remark that Theorem \ref{local duality} implies  the following well known fact,
%%%\begin{fact}\label{fact}
%%%\begin{equation*}
%%%\mbox{H}^i_p(\mathcal{F}^{\bullet})=0 \mbox { if and only if }
%%%\mathcal{E}xt^{-i}(\mathcal{F}^{\bullet}, \omega _R ^{\bullet})=0.
%%%\end{equation*}
%%%\end{fact}
%%%\begin{proof}
%%%We only show the only if statement since the other way is trivial. Suppose
%%%that $\mathcal{E}xt^{-i}(\mathcal{F}^{\bullet}, \omega _R ^{\bullet}) \neq 0.$
%%%Then the vector space $\mathcal{E}xt^{-i}(\mathcal{F}^{\bullet}, \omega _R ^{\bullet})
%%%\otimes R/p$ has a nontrivial map to $k$ by mapping a basis to $1_k$.
%%%But by definition of injective hull, $k \subset I$ so we have a non zero map
%%%\begin{equation*}
%%%\mathcal{E}xt^{-i}(\mathcal{F}^{\bullet}, \omega _R ^{\bullet})
%%%\rightarrow
%%%\mathcal{E}xt^{-i}(\mathcal{F}^{\bullet}, \omega _R ^{\bullet})\otimes R/p
%%%\rightarrow
%%%I.
%%%\end{equation*}
%%%In particular, $\mbox{H}^i_p(\mathcal{F}^{\bullet})\neq 0$.
%%\end{proof}

\begin{lemma}(cf.~\cite[Lemma 3.1]{Kovacs})\label{dualizing sheaf}
Given a closed point $x \in \Sigma$, let $\mathcal{O}_{\Sigma_x }$ be the local ring of $\mathcal{O}_{\Sigma}$ at $x$ and $\Sigma _x=\text{Spec} \mathcal{O}_{\Sigma_x}$ be the local scheme.
Then we have the following equation
\begin{equation*}
(\omega ^{\bullet}_{\Sigma})\otimes \mathcal{O}_{\Sigma _x} \iso \omega ^{\bullet}_{\Sigma_{x}},
\end{equation*}
where $\omega ^{\bullet}$ denotes dualizing complex.
\end{lemma}

\section{Du Bois Property}
In this section, we prove under Assumption \ref{assumption}, that $\Sigma(X,L)$ is always DB and give a partial answer to Question \ref{Question}.
It turns out that both results are related to the positivity property of $\shf{N}^*_{F_x/\mathbb{P}(E_{L})}$, the conormal bundle to $F_x $ in $\mathbb{P}(E_{L})$.

\begin{lemma}(\cite[Lemma 2.2]{Ullery14})\label{ullery}
Let $n=\dim X$. Assume $L$ is a $3$-very ample line bundle on $X$. Then for all $x\in X$,
\begin{equation*}
    \shf{N}^*_{F_x/\mathbb{P}(E_L)}\iso \sshf{F_x}^{\oplus n}\oplus b^*_xL(-2E_x),
\end{equation*}
where $F_x\iso Bl_xX$ and $E_x$ is the exceptional divisor.
\end{lemma}

The technical result below will be used in many places
later on.

\begin{proposition}\label{vanishing of higher direct images}
$R^it_*\mathcal{O}_{\mathbb{P}(E_L)}(-\Phi)=0$\: for all $i>0$.
\end{proposition}
\begin{proof}
Since $t:\mathbb{P}(E_L) \rightarrow \Sigma$ is an isomorphism over
$\Sigma \setminus X $, it suffices to prove the statement
at any closed point $x \in X$.

By the theorem on formal functions (cf.~\cite[III. 11]{Hartshorne77}), one has the isomorphism
\begin{equation*}
   R^it_*\widehat{\mathcal{O}}_{\mathbb{P}(E _L)}(-\Phi)_x\xrightarrow{\sim}\lim_{\substack{\leftarrow \\ k}} H^i(F_x, \mathcal{O}_{\mathbb{P}(E _L)}(-\Phi)\otimes\mathcal{O}_{\mathbb{P}(E _L)}/{\is{F_x}^k}),
\end{equation*}
so it suffices to show that $H^i(F_x, \mathcal{O}_{\mathbb{P}(E _L)}(-\Phi)\otimes\mathcal{O}_{\mathbb{P}(E_L)}/{\is{F_x}^k})=0$ for all $k>0$.

To this end we do induction on $k$. For $k=1$,
\begin{equation*}
   \mathcal{O}_{\mathbb{P}(E_L)}(-\Phi)\otimes\mathcal{O}_{\mathbb{P}(E_L)}/{\is{F_x}}\iso\shf{N}^*_{\Phi/{\mathbb{P}(E_L)}}\big|_{F_x}\iso b^*_xL(-2E_x),
\end{equation*}
where the last isomorphism is by \cite[p. 8]{Ullery14}, and $E_x$ is the exceptional divisor of the blowing-up of $X$ at $x$,  $b_x:F_x\rightarrow X$.

We now argue in case $n\ge 2$, but it is evident that the statement below is also valid for $n=1$.

Recalling Assumption \ref{assumption} for $L$, then for $j\ge 1$, we have
\begin{equation*}
    {b_x}^*L^j(-2jE_x))\iso \omega_{F_x}\otimes P^{n+1}\otimes Q,
\end{equation*}
where
\begin{eqnarray*}
% \nonumber to remove numbering (before each equation)
  P&=&b^*_xA^2(-E_x), \quad\text{and}  \\
   Q&=&\paren{\omega_{F_x}\otimes P^{n+1}\otimes b^*_xB}^{j-1}\otimes b^*_xB .
\end{eqnarray*}
It is well known that $P$ is very ample, so $\omega_{F_x}\otimes P^{n+1}\otimes b^*_xB$ is very ample, for  instance by  \cite[p.~57]{EL93}. It follows that $Q$ is nef and hence $P^{n+1}\otimes Q$ is ample.
Therefore by Kodaira vanishing, we have
\begin{equation}\label{vanishing}
        H^i(F_x, {b_x}^*L^j(-2jE_x))=0 \quad \text{for all}\: i>0.
\end{equation}
In particular, the vanishing above for $j=1$ is the desired vanishing for $k=1$.

For any $k>1$, consider the exact sequence
{\footnotesize
\begin{equation*}
    0\rightarrow{\mathcal{O}_{\mathbb{P}(E_L)}(-\Phi)\otimes{\is{F_x}^{k}}/{\is{F_x}^{k+1}}}\rightarrow{\mathcal{O}_{\mathbb{P}(E _L)}(-\Phi)\otimes\mathcal{O}_{\mathbb{P}(E _L)}/{\is{F_x}^{k+1}}}\rightarrow{\mathcal{O}_{\mathbb{P}(E_L)}(-\Phi)\otimes\mathcal{O}_{\mathbb{P}(E_L)}/{\is{F_x}^k}}\rightarrow 0.
\end{equation*}
}
We observe that
\begin{equation}\label{symmetric product}
    \mathcal{O}_{\mathbb{P}(E_L)}(-\Phi)\otimes{\is{F_x}^{k}}/{\is{F_x}^{k+1}}\iso b^*_xL(-2E_x)\otimes S^k\shf{N}^*_{F_x/\mathbb{P}(E_{L})}.
\end{equation}
By Lemma \ref{ullery},
\begin{equation*}
   S^k\shf{N}^*_{F_x/\mathbb{P}(E_L)}\iso \bigoplus^k_{j=0}\bigoplus^{{n+k-j-1\choose n-1}}b^*_xL^{j}(-2jE_x),
\end{equation*}
so (\ref{symmetric product}) is a direct sum with summands
\begin{equation*}
b_x^*L^{j+1}(-2(j+1)E_x), \quad \quad 0\le j \le k.
\end{equation*}
Then it is immediate by (\ref{vanishing}) that
\begin{equation*}
    H^i(F_x, \mathcal{O}_{\mathbb{P}(E_L)}(-\Phi)\otimes{\is{F_x}^{k}}/{\is{F_x}^{k+1}})=0,
\end{equation*}
which completes the proof together with the induction hypothesis.
\end{proof}

We need the following criterion for DB singularities.

\begin{theorem} \label{KK}(\cite[Theorem 1.6]{KK})
Let $f:Y\rightarrow X$  be a proper morphism between reduced schemes of finite
type over $\mathbb{C}$.
Let $W\subseteq X$ and $F:=f^{-1}(W)\subseteq Y$ be closed reduced subschemes with
ideal sheaves $\is{W/X}$ and $\is{F/Y}$.
Assume that the natural map
\begin{equation*}
\rho :\is{W/X}\longrightarrow \mathcal{R}f_*\is{F/Y}
\end{equation*}
admits a left inverse $\rho '$, that is, $\rho'\circ\rho =id_{\is{W/X}}$. Then if $Y,F$ and $W$ all have DB singularities, then so does $X$.
\end{theorem}

Let $V$ be a subspace of $H^0(X, L)$. Given a variety $X\subset\mathbb{P}(V)$, denote the secant variety by $\Sigma(X, V)$. The notation $\Sigma(X, L)$ coincides with $\Sigma(X, H^0(X, L))$. Theorem \ref{main theorem} can be slightly strengthened as follows.

\begin{theorem}\label{Du Bois}
Let $X\subset\mathbb{P}(V)$ be a smooth projective variety of dimension $n$ such that $L=\sshf{\mathbb{P}(V)}(1)|_X$ satisfies Assumption \ref{assumption}. Suppose $\Sigma(X, V)$ is normal. Then $\Sigma(X, V)$ has Du Bois singularities.
\end{theorem}

\begin{proof}
We first consider the special case that $V=H^0(X, \sshf{\mathbb{P}(V)}(1)|_X)$.
We claim that $t_*\sshf{\mathbb{P}(E_{L})}(-\Phi)\cong \is{X/{\Sigma}}$, where $\is{X/{\Sigma}}$ denotes the ideal sheaf of $X$ in $\Sigma$.
By the assumption on $L$, the main result in \cite{Ullery14} implies that $t_*\sshf{\mathbb{P}(E_{L})}=\mathcal{O}_{\Sigma}$.
So from the exact sequence
\begin{equation*}
    \ses{\sshf{\mathbb{P}(E_{L})}(-\Phi)}{\sshf{\mathbb{P}(E_{L})}}{\sshf{\Phi}},
\end{equation*}
we see that the claim follows from the fact $\mathcal{O}_X\iso q_*\mathcal{O}_{\Phi}$, see Lemma \ref{higher direct images of structure sheaf}.

With Proposition \ref{vanishing of higher direct images}, this implies that
$\mathcal{R}t_*\sshf{\mathbb{P}(E_{L})}(-\Phi)\cong \is{X/{\Sigma}}$ in the derived category.
In particular, the following splitting sequence holds
\begin{equation*}
\is{X/{\Sigma}}\rightarrow \mathcal{R}t_*\sshf{\mathbb{P}(E_{L})}(-\Phi) \rightarrow \is{X/{\Sigma}}.
\end{equation*}
Since $X$, $\Phi$ and $\mathbb{P}(E_{L})$ are all smooth, $\Sigma(X, L)$ has Du Bois singularities by Theorem \ref{KK}.

In general let $V\subseteq H^0(X, L)$ be a very ample linear system. For any $\xi\in X^{[2]}$, one has the natural commutative diagram
$$\xymatrix{
V \ar@{->>}[dr]\ar@{^{(}->}[d] & \\
H^0(X, L)\ar@{->>}[r] & H^0(X, L|_{\xi}).}$$

There exists a projective space $\Lambda\subseteq \mathbb{P}(H^0(X, L))$ such that $\Lambda\cap \mathbb{P}(V)=\emptyset$ and $\dim\Lambda+\dim \mathbb{P}(V)=\dim\mathbb{P}(H^0(X, L))-1$. Then one gets a projection $p: \mathbb{P}(H^0(X, L))\backslash{\Lambda}\rightarrow \mathbb{P}(V)$, which induces a morphism $p_{\Sigma}$, as shown below

$$\xymatrix{
 \Sigma(X, L) \ar[d]_{p_{\Sigma}}\ar@{^{(}->}[r] & \mathbb{P}(H^0(X, L))\backslash{\Lambda}\ar[d]_{p}\\
\Sigma(X, V) \ar@{^{(}->}[r]& \mathbb{P}(V).}$$

Take any point $p\in \mathbb{P}(V)$, let $\Lambda_p$ be the linear span by $\Lambda$ and $p$. Then set theoretically $p^{-1}_{\Sigma}(p)=\Sigma(X, L)\cap \Lambda_p$. Since $\Lambda$ is a hyperplane of $\Lambda_p$, $\Sigma(X, L)\cap \Lambda_p$ is a finite set of points, for otherwise $\Sigma(X, L)\cap\Lambda\neq\emptyset$. This indicates that $p_{\Sigma}$ is a finite morphism. Moreover $p_{\Sigma}$ is surjective. Applying \cite[Corollary 2.4]{KK} and using the fact that $\Sigma(X, L)$ is DB and the assumption that $\Sigma(X, V)$ is normal, we deduce that $\Sigma(X, V)$ is DB.
\end{proof}

Next we show that in some special cases, $\Sigma(X,L)$ has singularities more directly related to the minimal model program. For example $X=\mathbb{P}^2$, satisfies the hypothesis in the following theorem.

\begin{theorem}\label{log canonical}
If for any $x\in X$, $F_x\iso Bl_x X$ is a Fano variety, then on $\Sigma(X,L)$ there exists a boundary $\Delta$ such that $(\Sigma(X,L), \Delta)$ is a log canonical pair.
\end{theorem}
\begin{proof}
For any $x\in X$, $F_x$ is a fiber of $q: \Phi \rightarrow X$, so by adjunction formula we have
\begin{equation*}
(K_{\mathbb{P}(E_L)}+\Phi)\mid _{F_x}=K_{F_x}.
\end{equation*}
In other words, $-(K_{\mathbb{P}(E_L)}+\Phi)$ is relative ample over $\Sigma(X, L)$. (Recall that $t: \mathbb{P}(E_L)\rightarrow \Sigma(X,L)$ is isomorphic outside of $\Phi$.)

Pick an ample line $H$ on $\Sigma(X,L)$ such that $N:=-(K_{\mathbb{P}(E_L)}+\Phi)+t^*H$ is ample.
For $m\gg 0$, we can assume $mN\sim D$, where $D$ is a reduced divisor intersecting $\Phi$ transversely. Let $\Delta '=\Phi +\frac{1}{m}D$. We see that
\begin{equation*}
K_{\mathbb{P}(E_L)}+\Delta '\sim _{t, \mathbb{Q}}0.
\end{equation*}
Since $\Phi$ is a smooth divisor on $\mathbb{P}(E_L)$, $\Delta'$ is a linear combination of simple normal crossing divisors with coefficients at most one. So $\paren{\mathbb{P}(E_L), \Delta '}$ is a log canonical pair. Then \cite[Lemma 1.1]{Fujino12} implies that there exists a boundary $\Delta$ such that $(\Sigma(X,L), \Delta)$ is a log canonical pair.
\end{proof}

 %%%%%%%%%%%%%Non DuBois examples

\section {Non Du Boisness under weaker positivity}

In this section we consider the situation when the embedding line bundle $L$ fails to satisfy Assumption \ref{assumption}.
We show that the secant variety may still be normal, but not Du Bois anymore.

%%%%%%%%%% normal even if  L doesn't satisfy assumption 2.1 BEGIN%%%%%%%%

\subsection{Example of normal secant varieties}

The example below shows that Assumption \ref{assumption} is not necessary for the normality of secant varieties.

\begin{example}\label{normality under weak assumption}
Let $X\subset \mathbb{P}:=\mathbb{P}^{n+1}$ be an $n$-dimensional smooth hypersurface of degree $d$. By adjunction, $\omega_X\iso \sshf{X}(d-n-2)$. Fix an integer $k\ge 3$ and put $L=\sshf{X}(k)\iso \omega_X\otimes\sshf{X}(k-d+n+2)$, which is $3$-very ample.
\vspace{0.1cm}

Claim: $\Sigma(X, L)$ is normal provided that $k\ge \frac{d+3}{2}$

\begin{proof}
By \cite[Theorem D]{Ullery14}, it is sufficient to show that for each $x\in X$, $b^*_xL(-2E_x)$ is normally generated on the blowup $\text{Bl}_xX$. By pushing it down to $X$, it amounts to showing for each positive integer $r$
\begin{equation*}
    H^0(X, L\otimes\frak{m}^2_x)\otimes H^0(X, L^r\otimes\frak{m}^{2r}_x)\rightarrow H^0(X, L^{r+1}\otimes\frak{m}^{2(r+1)}_x )
\end{equation*}
is surjective.

To this end, consider the commutative diagram
$$\xymatrix{
H^0(\mathbb{P}, \sshf{\mathbb{P}}(k)\otimes\frak{m}^2_x )\otimes H^0(\mathbb{P}, \sshf{\mathbb{P}}(rk)\otimes\frak{m}^{2r}_x)\ar[r]\ar[d] & H^0(\mathbb{P}, \sshf{\mathbb{P}}((r+1)k)\otimes\frak{m}^{2(r+1)}_x )\ar[d]\ar[d]^{\alpha_{r+1}}\\
 H^0(X, L\otimes\frak{m}^2_x)\otimes H^0(X, L^r\otimes\frak{m}^{2r}_x) \ar[r] & H^0(X, L^{r+1}\otimes\frak{m}^{2(r+1)}_x ).}$$
After a change of coordinate for $\mathbb{P}$, $x=[1:0:\cdots :0]$ and a basis of the vector space $H^0(\mathbb{P}, \sshf{\mathbb{P}}(rk)\otimes\frak{m}^{2r}_x)$ is given by the set of monomials
\begin{equation*}
    \{x^{e_0}_0\cdots x^{e_{n+1}}_{n+1}\quad \big| \sum_i{e_i}=rk, e_0\le kr-2r\},
\end{equation*}
then it is straightforward to check that the top multiplication map is surjective. On the other hand,
if $k\ge\frac{d+3}{2}$, then $(r+1)k-d\ge 2(r+1)-1$. Applying Kawamata-Viehweg vanishing on the blowup of $\mathbb{P}$ at $x$ and noting that $\sshf{\mathbb{P}}(1)\otimes \frak{m}_x$ is globally generated, one obtains that
\begin{equation*}
    H^1(\mathbb{P}, \sshf{\mathbb{P}}((r+1)k-d)\otimes \frak{m}^{2(r+1)}_x)=0.
\end{equation*}
Using the standard exact sequence
\begin{equation*}
    \ses{\sshf{\mathbb{P}}(-d)}{\sshf{\mathbb{P}}}{\sshf{X}},
\end{equation*}
we deduce that the vertical map $\alpha_{r+1}$ is surjective, and the assertion follows.
\end{proof}
\end{example}
%%%%%%%%%% normal even if  L doesn't satisfy assumption 2.1 BEGIN%%%%%%%%

%%%%%%%%%if and only if for Du Bois BEGIN%%%%%%%%%%%

\subsection{Examples of non Du Bois secant varieties}

In this subsection we show that if the embedding line bundle $L$ does not satisfy  assumption \ref{assumption}, then
the secant variety $\Sigma(X,L)$ may not be DB, even it is normal.
In fact we prove a necessary and sufficient condition for a secant variety to be DB.
The condition is similar to Steenbrink's characterization of isolated DB singularities (cf.~\cite[Corollary 1.5]{Ishii}, \cite{Steenbrink83}).

\begin{theorem} \label{iff}
Suppose $L$ is 3-very ample and $\Sigma$ is seminormal. Then $\Sigma$ is Du Bois if and only if $R^it_*\sshf{\mathbb{P}(E_L)}(-\Phi)=0$ for all $i> 0$.
\end{theorem}
\begin{proof}
By the definition of Du Bois pair, we have the following diagram
$$\xymatrix{
{\underline{\Omega}^0_{\Sigma , X}} \ar[d] ^{\alpha}\ar[r] &{\underline{\Omega}^0_{\Sigma }} \ar[d]^{\beta} \ar[r]^{\psi} &{\underline{\Omega}^0_{ X}}\ar[d]\ar[r]^{+1}&& \\
R  t_*{\underline{\Omega}^0_{\mathbb{P}(E_L) , \Phi}} \ar[r] &R  t_*{\underline{\Omega}^0_{\mathbb{P}(E_L) }} \ar[r] &R  t_*{\underline{\Omega}^0_{\Phi} } \ar[r]^{+1} &&} $$
where the map $\alpha$ is induced from the other two vertical natural maps.
By \cite[Lemma 2.1]{KK}, there is an object $D$ fitting in the distinguished triangle
\begin{equation} \label{D}
D\longrightarrow {\underline{\Omega}^0_{ X}} \oplus R  t_*{\underline{\Omega}^0_{\mathbb{P}(E_L) }} \longrightarrow R  t_*{\underline{\Omega}^0_{\Phi} } \overset{+1}{\longrightarrow}.
\end{equation}
In fact the object $D$ is just the mapping cone shifted to the left by one.
As a result we can define a map $\delta : {\underline{\Omega}^0_{\Sigma }}  \longrightarrow D$ by $\psi \oplus \beta.$

Moreover since $t$ is an isomorphism outside of $X$, we have the following distinguished triangle (see for instance \cite[Theorem 6.5 (10)]{kollar:book}),
\begin{equation*}
    {\underline{\Omega}^0_{\Sigma}}\rightarrow {\underline{\Omega}^0_X\oplus Rt_* \underline{\Omega}^0_{\mathbb{P}(E_L)}}\longrightarrow {Rt_* \underline{\Omega}^0_{\Phi}}\xrightarrow{+1}.
\end{equation*}
Combining with (\ref{D}), we obtain the following commutative diagram
$$\xymatrix{
{\underline{\Omega}^0_{\Sigma}} \ar[d] ^{\delta}\ar[r] &{\underline{\Omega}^0_X\oplus Rt_* \underline{\Omega}^0_{\mathbb{P}(E_L)}}\ar[d]^{=} \ar[r] &{Rt_* \underline{\Omega}^0_{\Phi}}\ar[d]^{=}\ar[r]^{+1}& \\
D \ar[r] &{\underline{\Omega}^0_X\oplus Rt_* \underline{\Omega}^0_{\mathbb{P}(E_L)}}\ar[r] &{Rt_* \underline{\Omega}^0_{\Phi}}\ar[r]^{+1} &} $$
In particular, we see that $\delta $ is a quasi-isomorphism.
Then by in \cite[Lemma 2.1]{KK}, $\alpha$ is a quasi-isomorphism.

Since $\mathbb{P}(E_L)$ and $\Phi$ are both smooth, $\underline{\Omega} ^0_{\mathbb{P}(E_L) , \Phi} \cong \mathcal{O}_{\mathbb{P}(E_l)}(-\Phi)$.
By the isomorphism $\alpha$, the sequence
\begin{equation*}
\underline{\Omega} ^0_{\Sigma,X}\longrightarrow \underline{\Omega}^0_{\Sigma } \longrightarrow \underline{\Omega}^0_X \overset{+1}{\longrightarrow}
\end{equation*}
becomes
\begin{equation*}
Rt_*\mathcal{O}_{\mathbb{P}(E_l)}(-\Phi) \longrightarrow \underline{\Omega}^0_{\Sigma } \longrightarrow \underline{\Omega}^0_{X }\overset{+1}{\longrightarrow},
\end{equation*}
which yields the following long exact sequence
\begin{equation*}
0\rightarrow t_*\mathcal{O}_{\mathbb{P}(E_l)}(-\Phi) \rightarrow h^0(\underline{\Omega}^0_{\Sigma })\rightarrow \mathcal{O}_X\rightarrow R^1t_*\mathcal{O}_{\mathbb{P}(E_l)}(-\Phi) \rightarrow h^1(\underline{\Omega}^0_{\Sigma }) \rightarrow \cdots .
\end{equation*}

On the other hand, since $\Sigma$ is seminormal, the natural map $\sshf{\Sigma}\rightarrow h^0(\underline{\Omega}^0_{\Sigma })$ is an isomorphism.
Therefore the map $h^0(\underline{\Omega}^0_{\Sigma })\rightarrow\sshf{X}$ is surjective. Besides note that $h^i(\underline{\Omega}^0_{X })=0$ for all $i > 0$, so
\begin{equation*}
R^it_*\mathcal{O}_{\mathbb{P}(E_l)}(-\Phi) \iso h^i(\underline{\Omega}^0_{\Sigma })
\end{equation*}
for all $i>0$.
In particular, $\Sigma$ is Du Bois if and only if $R^it_*\sshf{\mathbb{P}(E_L)}(-\Phi)=0$ for all $i>0$.
\end{proof}

\begin{remark}
The argument in the proof of Theorem \ref{iff} is essentially the same as the proof of \cite[Theorem 2.1]{GK},
where the case of cone singularities is considered.
\end{remark}

Next we provide examples indicating that there exist secant varieties to smooth varieties of arbitrary dimension which are normal but not DB.

\begin{proposition}\label{nonvanishing of higher direct image}
Let $L$ be a 3-very ample line bundle on $X$ of dimension $n$. If $H^n(X, L)\neq 0$, then $R^nt_*\sshf{\mathbb{P}(E_L)}(-\Phi)\neq 0$.
\end{proposition}

\begin{proof}
By the theorem on formal functions (cf.~\cite[III. 11]{Hartshorne77}), one has the isomorphism
\begin{equation*}
   R^nt_*\widehat{\mathcal{O}}_{\mathbb{P}(E _L)}(-\Phi)_x\xrightarrow{\sim}\lim_{\substack{\leftarrow \\ k}} H^n(F_x, \mathcal{O}_{\mathbb{P}(E _L)}(-\Phi)\otimes\mathcal{O}_{\mathbb{P}(E _L)}/{\is{F_x}^k}).
\end{equation*}
For $k=1$,
\begin{eqnarray*}
    &&H^n(F_x, \mathcal{O}_{\mathbb{P}(E _L)}(-\Phi)\otimes\mathcal{O}_{\mathbb{P}(E _L)}/{\is{F_x}})\\
    &\iso & H^n(F_X, b^*_xL(-2E_x))\\
    &\iso & H^n(X, L)\\
    &\neq & 0.
\end{eqnarray*}
Because dim$F_x=n$,
for each $k>0$ the natural map
\begin{equation*}
    H^n(F_x, \mathcal{O}_{\mathbb{P}(E _L)}(-\Phi)\otimes\mathcal{O}_{\mathbb{P}(E _L)}/{\is{F_x}^{k+1}})\rightarrow H^n(F_x, \mathcal{O}_{\mathbb{P}(E _L)}(-\Phi)\otimes\mathcal{O}_{\mathbb{P}(E _L)}/{\is{F_x}^k})
\end{equation*}
is surjective. Therefore any nonzero element in $H^n(F_x, \mathcal{O}_{\mathbb{P}(E _L)}(-\Phi)\otimes\mathcal{O}_{\mathbb{P}(E _L)}/{\is{F_x}})$ lifts to a nonzero element in
\begin{equation*}
    \lim_{\substack{\leftarrow \\ k}} H^n(F_x, \mathcal{O}_{\mathbb{P}(E _L)}(-\Phi)\otimes\mathcal{O}_{\mathbb{P}(E _L)}/{\is{F_x}^k}).
\end{equation*}
\end{proof}

\begin{example}
Let $X$ be a smooth projective curve with Clifford index\\ $\text{Cliff}(X)\ge 3$. According to \cite[Corollary B]{Ullery14}, $\Sigma(X, \omega_X)$ is normal. On the other hand, $H^1(X, \omega_X)\neq 0$, so Theorem \ref{iff} and Proposition \ref{nonvanishing of higher direct image} imply that $\Sigma(X, \omega_X)$ is not DB.
\end{example}

\begin{example}
Again let $X\subset \mathbb{P}^{n+1}$ be a smooth hypersurface of degree $d\ge n+2+k$ and $L=\sshf{X}(k)$ with $k\ge 3$. Then
\begin{equation*}
    H^n(X, L)\iso H^{n+1}(\prj{n+1}, \sshf{\prj{n+1}}(k-d))\neq 0,
\end{equation*}
which shows that $\Sigma(X, L)$ is not DB. If in addition $k\ge \frac{d+3}{2}$, we have $\Sigma(X, L)$ is normal by Example \ref{normality under weak assumption}. These conditions can be achieved simultaneously for any fixed $n$ by, for instance, taking an even number $d\ge 2(n+4)$ and $k=\frac{d+4}{2}$.
\end{example}

\section{On Cohen-Macaulayness}
In this section, we calculate the depth of local rings of $\Sigma$ which measures how far $\Sigma$ is from being Cohen-Macaulay, and also study the sheaf $t_*\omega_{\mathbb{P}(E_L)}$ using the language of derived category.

Since we are dealing with the projective case, a basic lemma for our purpose is
\begin{lemma}(\cite[Lemma 2.3]{AH12})
Let $X$ be a projective scheme over an algebraically closed field $K$ of pure dimension $n$ with an ample divisor $D$. Let $\shf{F}$ be a coherent sheaf on $X$ such that support of $\shf{F}$ has pure dimension $n$. Then
$\text{depth}_x\shf{F}\ge k$ for all closed point $x\in X$ if and only if $H^i(X, \shf{F}(-rD))=0$ for all $i<k$ and $r\gg 0$.
\end{lemma}

With the lemma above, we will prove
\begin{theorem}\label{CM}
If $L$ is defined as Assumption \ref{assumption}, then
\begin{equation*}
    \text{depth}_x\Sigma(X, L)=n+2+\max\{i \:|\: i\le n-1, \text{and}\; H^j(X, \sshf{X})=0,  1\le j\le i\}.
\end{equation*}
Here we adopt the convention that if the set $\{i \:|\: i\le n-1, \text{and}\; H^j(X, \sshf{X})=0,  1\le j\le i\}$ is empty, then the max is 0.
\end{theorem}

\begin{corollary}
If $X$ is a curve , then $\Sigma(X,L)$ is Cohen-Macaulay.\qed
\end{corollary}
\begin{remark}
Sidman and Vermeire \cite{SV09} proved a stronger result that for a curve of genus $g$ embedded as linear normal curve by a line bundle of degree $d\ge 2g+3$, $\Sigma$ is arithmetically Cohen-Macaulay.
\end{remark}

\begin{corollary}
If $X$ can be embedded in a projective space as a scheme-theoretic complete intersection, then for sufficiently ample $L$, $\Sigma(X, L)$ is CM.
\end{corollary}
\begin{proof}
Since $X$ is a complete intersection, $H^i(X, \sshf{X})=0$ for $1\le i\le n-1$. Theorem \ref{CM} implies $\text{depth}_x\Sigma(X, L)$=2n+1.
\end{proof}

\begin{lemma}\label{lemma2}
For all integers $i, r$,
\begin{equation*}
   H^i\paren{\Sigma, \is{X/{\Sigma}}(-r)}\iso H^i\paren{\mathbb{P}(E_{L}), \sshf{\mathbb{P}(E_{L})}(-r)\otimes\sshf{\mathbb{P}(E_{L})}(-\Phi)}.
\end{equation*}
\end{lemma}
\begin{proof}
Since $R^jt_*(\sshf{\mathbb{P}(E_{L})}(-r)\otimes\sshf{\mathbb{P}(E_{L})}(-\Phi))\iso \sshf{\Sigma}(-r)\otimes R^it_*\sshf{\mathbb{P}(E_{L})}(-\Phi)=0$ for $j>0$, by Proposition \ref{vanishing of higher direct images}, we have
\begin{equation*}
   H^i\paren{\mathbb{P}(E_{L}), \sshf{\mathbb{P}(E_{L})}(-r)\otimes\sshf{\mathbb{P}(E_{L})}(-\Phi)} \iso  H^i\paren{\Sigma, t_*\sshf{\mathbb{P}(E_{L})}(-\Phi)\otimes\sshf{\Sigma}(-r)}.
  \end{equation*}
To finish the proof, we use the fact $t_*\sshf{\mathbb{P}(E_{L})}(-\Phi)\iso \is{X/{\Sigma}}$, as shown in the proof of Theorem \ref{Du Bois}.
\end{proof}

\begin{lemma}\label{lemma3}
Let $r\gg 0$. Then
\begin{equation*}
    H^i(\Phi, \sshf{\Phi}(-r))=\left\{\begin{array}{ll}
    0 & \text{if}\quad i<n, \\
    H^n(X, L^{-r})\otimes H^{i-n}(X, \sshf{X}) & \text{if}\quad n\le i\le 2n.
    \end{array}\right.
\end{equation*}
\end{lemma}
\begin{proof}
Since $L$ is ample, we have that if $b<n$ and $r\gg 0$, $H^{b}(X, R^aq_*\sshf{\Phi}\otimes L^{-r})=0$ by Serre vanishing. It follows that the Leray spectral sequence
\begin{equation*}
    E^{a, b}_2=H^b(X, R^aq_*\sshf{\Phi}\otimes L^{-r})\Rightarrow H^{a+b}\paren{\Phi, \sshf{\Phi}(-r)}
\end{equation*}
degenerates at the level $E_2$. We see that if $i<n$, $H^i(\Phi, \sshf{\Phi}(-r))=0$ for $r\gg 0$. For $n\le i\le 2n$,
\begin{equation*}
   H^i(\Phi, \sshf{\Phi}(-r))\iso H^n(X, R^{i-n}q_*\sshf{\Phi}\otimes L^{-r})\iso H^n(X, L^{-r})\otimes H^{i-n}(X, \sshf{X}),
\end{equation*}
where the last isomorphism is by Lemma (\ref{higher direct images of structure sheaf}).
\end{proof}

\begin{proof}[Proof of Theorem \ref{CM}]
Since $\sshf{\mathbb{P}(E_{L})}(1)=t^*\sshf{\Sigma}(1)$, the tautological line bundle $\sshf{\mathbb{P}(E_{L})}(1)$ is big and nef. Therefore by Kawamata-Viehweg vanishing,
\begin{equation}\label{KV vanishing}
   H^i(\sshf{\mathbb{P}(E_{L})}(-r))=0 \quad\text{for all $i<2n+1$ and $r>0$}.
\end{equation}

Consider the exact sequence
\begin{equation}\label{ses}
    \ses{\sshf{\mathbb{P}(E_{L})}(-r)\otimes \sshf{\mathbb{P}(E_{L})}(-\Phi)}{\sshf{\mathbb{P}(E_{L})}(-r)}{\sshf{\Phi}(-r)},
\end{equation}
where $\sshf{\Phi}(1)=q^*L$. We obtain that for $i<2n+1$,
\begin{eqnarray*}
    &&H^i(\Sigma, \is{X/{\Sigma}}(-r))\\
    &\iso & H^i\paren{\sshf{\mathbb{P}(E_{L})}(-r)\otimes\sshf{\mathbb{P}(E_{L})}(-\Phi)} \quad\text{by Lemma}\:\ref{lemma2}\\
    &\iso & H^{i-1}(\sshf{\Phi}(-r)) \hspace{3.0cm}\text{by}\:(\ref{KV vanishing}),(\ref{ses})\: (H^{-1}=0)\\
    &\iso & \left\{\begin{array}{ll}
    0 & \text{if}\quad i\le n, \\
    H^n(X, L^{-r})\otimes H^{i-1-n}(X, \sshf{X}) & \text{if}\quad n<i<2n+1,
    \end{array}\right.
\end{eqnarray*}
where the last isomorphism is by Lemma \ref{lemma3}.

Then from the exact sequence
\begin{equation*}
    \ses{\is{X/{\Sigma}}(-r)}{\sshf{\Sigma}(-r)}{\sshf{X}(-r)},
\end{equation*}
and applying Kodaira vanishing to $\sshf{X}(-r)$, we see
\begin{equation*}
    H^i(\sshf{\Sigma}(-r))=\left\{\begin{array}{ll}
    0 & \text{if}\quad i<n, \\
    H^n(X, L^{-r})\otimes H^{i-1-n}(X, \sshf{X}) & \text{if}\quad n+1<i<2n+1.
    \end{array}\right.
\end{equation*}
So to finish the proof, we shall show that $H^n(\sshf{\Sigma}(-r))=H^{n+1}(\sshf{\Sigma}(-r))=0$. In view of above results, both groups sit in the exact sequence
\begin{equation*}
    0\rightarrow H^n(\sshf{\Sigma}(-r))\rightarrow H^n(\sshf{X}(-r))\xrightarrow{\alpha} H^{n+1}(\Sigma, \is{X/{\Sigma}}(-r))\rightarrow H^{n+1}(\sshf{\Sigma}(-r))\rightarrow 0,
\end{equation*}
and hence we need to prove $\alpha$ is an isomorphism.

To this end, consider the natural commutative diagram
$$\xymatrix{
H^n(\sshf{\Phi}(-r))  \ar[r] & H^{n+1}\paren{\sshf{\mathbb{P}(E_{L})}(-r)\otimes\sshf{\mathbb{P}(E_{L})}(-\Phi)}\\
H^n(\sshf{X}(-r))\ar[u]\ar[r]^{\alpha} & H^{n+1}\paren{\Sigma, \is{X/{\Sigma}}(-r)}\ar[u] }.$$

Since the top and right column maps are isomorphisms, it is reduced to show that the natural map
$H^n(\sshf{X}(-r))\rightarrow H^n(\sshf{\Phi}(-r))$ is an isomorphism. Again, this is the case by the Leray spectral sequence and Serre vanishing.
\end{proof}

\begin{theorem}\label{pushforward of dualizing sheaf}
 $t_*\omega _{\mathbb{P}(E_{L})}\iso \omega_{\Sigma}$ if and only if $H^n(X, \mathcal{O}_X)=0$.
\end{theorem}
\begin{proof}
Assuming Claim \ref{claim} and
pushing forward the exact sequence
\begin{equation*}
0\rightarrow \omega_{\mathbb{P}(E_L)}\rightarrow\omega_{\mathbb{P}(E_L)}(\Phi)\rightarrow i_*\omega_{\Phi}\rightarrow 0
\end{equation*}
by $t$, we see that the statement of this theorem is equivalent to showing
\begin{equation*}
(t_*\circ i_*)\omega_{\Phi}=0 \mbox{ if and only if } H^n(X, \mathcal{O}_X)=0.
\end{equation*}
By virtue of Diagram \ref{Geometry of secant varieties}, it suffices to show that $q_*\omega_{\Phi}=0$ on $X$.
For a closed point $x \in X$, it is equivalent to showing $q_*\omega_{\Phi}\otimes k(x)=0$ by Nakayama's lemma. But since $q$ is flat by Lemma \ref{flat}, applying Grauert's theorem we have
\begin{equation*}
q_*\omega_{\Phi}\otimes k(x)\iso H^0(F_x, \omega _{\Phi}\mid _{F_x})\iso H^0(F_x, \omega_{F_x}).
\end{equation*}
Recall that $F_x$ is the blowing up of $X$ at $x$, so $H^0(F_x, \omega_{F_x})=H^0(X, \omega_X)$
 and the last cohomology group is zero if and only if $H^n(X, \sshf{X})=0$.
\end{proof}

\begin{claim}\label{claim}
\begin{equation*}
t_*\omega_{\mathbb{P}(E_{L})}(\Phi)\iso \omega_{\Sigma}.
\end{equation*}
\end{claim}
\begin{proof}
First note that there is a natural map $t_*\omega_{\mathbb{P}(E_{L})}(\Phi)\rightarrow\omega_{\Sigma}$ by the following natural maps
\begin{equation*}
t_*\omega_{\mathbb{P}(E_{L})}(\Phi)\rightarrow  j_*(t_*\omega_{\mathbb{P}(E_{L})}(\Phi)|_U)\rightarrow j_*(\omega_{\Sigma}|_U)\iso \omega_{\Sigma}.
\end{equation*}
Here $j: \Sigma \setminus X\rightarrow\Sigma$ is the open immersion and
the last isomorphism comes from the fact that the sheaf $\omega_{\Sigma}$ is reflexive.
So to prove the claim it suffices to prove that
\begin{equation*}
t_*\omega_{\mathbb{P}(E_{L})}(\Phi)_x\iso {\omega_{\Sigma}}_x
\end{equation*}
for every closed point $x\in \Sigma$.

From now on we fix a closed point $x\in X \subset \Sigma$ and denote by $\mathcal{O}_{\Sigma _x}$ the local ring of
$\Sigma$ at $x$. By Proposition \ref{vanishing of higher direct images} we have
\begin{equation*}
\mathcal{R} t_*\sshf{\mathbb{P}(E_{L})}(-\Phi)\iso t_*\sshf{\mathbb{P}(E_{L})}(-\Phi),
\end{equation*}
so applying the functor $\mathcal{R}Hom_{\Sigma}( \:, \omega _{\Sigma}^{\bullet})$  then localizing at $x$ and denoting the local ring of
$\Sigma$ at $x$ by $\mathcal{O}_{\Sigma _x}$,
 we get the isomorphisms of complexes
\begin{eqnarray*}
\mathcal{R}t_*\omega ^{\bullet}_{\mathbb{P}(E_{L})}(\Phi) \otimes
\mathcal{O}_{\Sigma _x} &\iso & \mathcal{R}Hom_{\Sigma}(\mathcal{R} t_*\sshf{\mathbb{P}(E_{L})}(-\Phi), \omega _{\Sigma}^{\bullet})\otimes
\mathcal{O}_{\Sigma _x}\\
&\iso & \mathcal{R}Hom_{\Sigma}( t_*\sshf{\mathbb{P}(E_{L})}(-\Phi), \omega _{\Sigma}^{\bullet})\otimes
\mathcal{O}_{\Sigma _x}\\
&\iso &\mathcal{R}Hom_{\Sigma_x}( t_*\sshf{\mathbb{P}(E_{L})}(-\Phi)_x, \omega _{\Sigma_x}^{\bullet}),
\end{eqnarray*}
where the first isomorphism follows from Grothendieck Duality theorem and the last
one is by Lemma \ref{dualizing sheaf}.
By taking the $-i$th cohomology we see that
\begin{equation*}
R^{2n+1-i}t_*\omega_{\mathbb{P}(E_{L})}(\Phi) \otimes
\mathcal{O}_{\Sigma _x} \iso \mathcal{E}xt^{-i}(t_*\sshf{\mathbb{P}(E_{L})}(-\Phi)_x, \omega _{\Sigma_x}^{\bullet}).
\end{equation*}
In particular, when $i=2n+1$, we get the isomorphism
\begin{equation} \label{new isomo 2}
t_*\omega_{\mathbb{P}(E_{L})}(\Phi) \otimes
\mathcal{O}_{\Sigma _x} \iso \mathcal{E}xt^{-(2n+1)}(t_*\sshf{\mathbb{P}(E_{L})}(-\Phi)_x, \omega _{\Sigma_x}^{\bullet}).
\end{equation}

Let $I$ be the injective hull of $\mathcal{O}_{\Sigma _x}$.
Apply $Hom(\:, I)$ to the right hand side of equation (\ref{new isomo 2}), then by (\ref{equ}) we have
\begin{eqnarray}\label{qq}
 &&Hom(\mathcal{E}xt^{-(2n+1)}(t_*\sshf{\mathbb{P}(E_{L})}(-\Phi)_x, \omega _{\Sigma_x}^{\bullet}), I)\\
\nonumber &\iso & H^{2n+1}_x(t_*\sshf{\mathbb{P}(E_{L})}(-\Phi)_x)\\
\nonumber&\iso & H^{2n+1}_x(\mathcal{O}_{\Sigma_x}),
\end{eqnarray}
where the second isomorphism comes from the following sequence
\begin{equation*}
0\rightarrow t_*\sshf{\mathbb{P}(E_{L})}(-\Phi)\iso \is{X/{\Sigma}}
\rightarrow \mathcal{O}_{\Sigma} \rightarrow \mathcal{O}_X\rightarrow 0
\end{equation*}
and the fact that  $H^i_x(\mathcal{O}_{X_x})=0$ for all $i>n$.

Apply $Hom(\:, I)$ to equation (\ref{qq}), by \cite[$\S$V Corollary 6.5]{Har66} we have an isomorphism of the completions at $x$
\begin{eqnarray*}
&&\mathcal{E}xt^{-(2n+1)}(t_*\sshf{\mathbb{P}(E_{L})}(-\Phi)_x, \omega _{\Sigma_x}^{\bullet})^{\wedge}_x\\
&\iso & Hom (H^{2n+1}_x(t_*\sshf{\mathbb{P}(E_{L})}(-\Phi)_x), I)\\
&\iso & Hom(H^{2n+1}_x(\mathcal{O}_{\Sigma_x}), I)\\
&\iso & \mathcal{E}xt^{-(2n+1)}(\mathcal{O}_{\Sigma_x},\omega _{\Sigma_x}^{\bullet} )^{\wedge}_x\\
&\iso & (\omega _{\Sigma_x}) ^{\wedge}_x.
\end{eqnarray*}
But by equation (\ref{new isomo 2}) there is  a natural map
\begin{equation*}
\mathcal{E}xt^{-(2n+1)}(t_*\sshf{\mathbb{P}(E_{L})}(-\Phi)_x, \omega _{\Sigma_x}^{\bullet})\iso t_*\omega_{\mathbb{P}(E_{L})}(\Phi)\otimes
\mathcal{O}_{\Sigma _x}\rightarrow \omega _{\Sigma_x},
\end{equation*}
so isomorphism of the completions implies the isomorphism
of the local rings, that is,
\begin{equation*}
\mathcal{E}xt^{-(2n+1)}(t_*\sshf{\mathbb{P}(E_{L})}(-\Phi)_x, \omega _{\Sigma_x}^{\bullet})\iso \omega _{\Sigma_x}.
\end{equation*}
Thus
\begin{equation*}
t_*\omega_{\mathbb{P}(E_{L})}(\Phi) \otimes
\mathcal{O}_{\Sigma _x} \iso \omega _{\Sigma_x}.
\end{equation*}
This completes the proof of the claim.
\end{proof}

\begin{remark}
We emphasize that validity of Theorem \ref{pushforward of dualizing sheaf} does not rely on $H^{i}(X, \shf{O}_X)=0$ for $1\le i \le {n-1}$, or equivalently Cohen-Macaulayness of $\Sigma$.
If $\Sigma$ is Cohen-Macaulay, then claim \ref{claim} and hence Theorem \ref{pushforward of dualizing sheaf}
follow immediately from \cite[Theorem 3.1]{KSS}, since we already proved that $\Sigma$ is Du Bois.
\end{remark}

%%The following corollary is an easy consequence of  Theorem \ref{pushforward of dualizing sheaf},
%%\begin{corollary}
%%Under our assumption \ref{assumption}.
%%Kawamata-Viewheg vanishing theorem holds on $\Sigma(X, L)$.
%%\end{corollary}

%%\begin{remark}
%%By Theorem \ref{Du Bois}, we know that in this situation $\Sigma(X, L)$ is Du Bois.
%%It is well known that Kodaira vanishing theorem holds on Du Bois singular varieties.
%%\end{remark}

\section{A multiplicity formula}

The following proposition shows that more positivity of $L$ may lead to worse singularities in a certain sense.

\begin{proposition}
The Samuel multiplicity of $\Sigma(X, L)$ at a closed point $x\in X$ is given by $L^n-2^n$.
\end{proposition}

\begin{proof}
Since the multiplicity $\mu$ at $x$ coincides with the top Segre class of $(\{x\}, \Sigma(X, L))$, which is invariant under a birational proper morphism (cf.~\cite[Chap. 4]{Fulton98}), we have
\begin{eqnarray*}
% \nonumber to remove numbering (before each equation)
   \mu_x\Sigma(X, L)&=& s_0(\{x\}, \Sigma(X, L)) \\
   &=& s_n(Bl_xX, \mathbb{P}(E_L))\\
   &=& s_n\paren{\shf{N}_{{Bl_xX}/\mathbb{P}(E_L)}}\\
   &=& (-1)^n (-b_x^*L+2E)^n\quad\quad\quad\text{by Lemma } \ref{ullery}\\
   &=& (-1)^n\sum^n_{i=0}{n \choose i}(-1)^i(b_x^*L)^i(2E)^{n-i}\\
   &=& L^n+2^n(-E)^n\quad\quad\quad\quad\quad\text{for}\; 0<i<n,  (b^*L)^i\cdot E^{n-i}=0\\
   &=& L^n-2^n,
\end{eqnarray*}
which completes the proof.
\end{proof}

\bibliography{Singularities_of_secant_varieties}{}
\bibliographystyle{amsplain}

\end{document}